\documentclass[12pt,a4paper]{amsart}
\usepackage{amsmath, amsthm, amssymb}
\usepackage{amsfonts}
\usepackage{amsthm}

\usepackage{amsmath}
\usepackage{amssymb}
\usepackage{enumitem,kantlipsum}
\newtheorem{theorem}{Theorem}[section]
\newtheorem{lemma}[theorem]{Lemma}

\newcommand{\R}{\mathbb R}

\usepackage[colorlinks=true, allcolors=black]{hyperref}

\theoremstyle{definition}

\newtheorem{example}[theorem]{Example}

\newtheorem{corollary}[theorem]{Corollary}
\newtheorem{que}[theorem]{Question}
\theoremstyle{remark}
\newtheorem{remark}[theorem]{Remark}

\numberwithin{equation}{section} \DeclareMathOperator{\Image}{Im}
\DeclareMathOperator{\trace}{tr} \DeclareMathOperator{\tra}{tr}

\DeclareMathOperator{\Oct}{\mathbb{O}}
\def\a{\alpha}
\def\id{\operatorname{id}}
\title[Images of polynomials on octonions]{The images of multilinear and semihomogeneous polynomials on the algebra of octonions}

\thanks{$^1 $Department of Mathematics, Bar-Ilan University, Ramat Gan 52100, Israel}
\thanks{$^2 $Department of Mathematics, Ariel university,  Ariel 40700, Israel}
\thanks{We would like to thank I.Shestakov and E.Plotkin for interesting and fruitful discussions regarding this paper.}

\author[A. Kanel-Belov]{Alexei Kanel-Belov$^1$}
\author[S. Malev]{Sergey Malev$^2$}
\author[C. Pines]{Coby Pines$^2$}
\author[L. Rowen]{Louis Rowen$^1$}
\email[Alexei Kanel-Belov]{kanelster@gmail.com}
\email[Sergey Malev]{sergeyma@ariel.ac.il}
\email[Coby Pines]{cobypinesdirac@gmail.com}
\email[Louis Rowen]{rowen@math.biu.ac.il}

\begin{document}

\begin{abstract}
The generalized L'vov-Kaplansky conjecture states that for any finite-dimensionl simple algebra $A$ the image of a multilinear polynomial on $A$ is a
vector space.
In this paper we prove it for the algebra of octonions $\Oct$ over a  field satisfying certain specified
conditions (in particular, we prove it for quadratically closed field and for field $\R$).
 In fact, we prove that the image set must be either
$\{0\}$, $ F $, the space of pure octonions $V$, or $\Oct$.
We discuss possible evaluations of semihomogeneous polynomials on $\Oct$ and of arbitrary polynomials on the corresponding Malcev algebra.
\end{abstract}

\maketitle

\section{Introduction} The question of possible  evaluations of a multilinear polynomial $p$ on a matrix
algebra was reputedly raised by Kaplansky. L'vov reformulated this
question, asking whether $\Image p$ is a vector subspace.

In recent years there has been considerable interest on the values
of a polynomial $p$ on an algebra, \cite{Br,BrK,Ch,Spe}, in
particular on matrix algebras  and related finite dimensional
algebras,
\cite{BMR1,BMR2,BMR3,BMR4,BMRY,DK,LT,MO,M1,M2,MP,Spe,W1,W2,WLB}.

Note  for a non-simple finite dimensional algebra that the
Kaplansky-L'vov problem can fail. In particular, it fails for a
Grassmann algebra of a linear space of dimension at least $4$, over
a field of characteristic $\neq 2$, which is finite dimensional but
not simple. Indeed, consider the multilinear polynomial
$p(x,y)=x\wedge y-y\wedge x$. In this case $e_1\wedge e_2=p(\frac 12
e_1,e_2)$ and $e_3\wedge e_4$ both belong to the image of $p$, but
their sum does not. Thus the image is not a vector space.

On the other hand the requirement of finite dimensionality is
important as well. Indeed, in \cite{ML} one  finds an example of an
(infinite dimensional) division algebra for which this conjecture
fails for the polynomial $p(x,y)= xy - yx.$

For a more detailed review see \cite{BMRY}. In this paper we
consider the octonians $\Oct = (Q,\a)$, a simple alternative
algebra, and prove, under a mild condition on~$\a$:

\medskip

{\bf Theorem~\ref{Oct}.} The image of a multilinear polynomial $p$
on an octonion algebra $\Oct$ is either $\{0\}$, $ F $, $V$ or
$\Oct$.

\bigskip

One can also get results for ``semihomogeneous'' polynomials.

\medskip

{\bf Theorem~\ref{semi}.} Let $p(x_1,\dots, x_m)$ be a
semihomogeneous polynomial with nonzero weighted degree $d$, i.e.,
letting $d_i$ be the degree of $x_i$, there exist weights
$w_1,w_2,\dots,w_m$ such that in each monomial
 $$w_1d_1+w_2d_2+\dots+w_md_m=d.$$ Then the evaluation of $p$ on
$\Oct$ is either $\{0\}$, or $F$, or V is Zariski dense in $\Oct$.

\bigskip

Finally, we also treat Mal'cev algebras:

{\bf Theorem~\ref{Mal}.} Let $p(x_1,\dots,x_m)$ be an arbitrary
polynomial in $m$ anticommutative variables. Then its evaluation on
$V$ is either $\{0\}$ or~$V$.

\section{Preliminaries} Throughout this paper, $F$ denotes  a
field of  characteristic $\neq 2$, and $A$ denotes a unital algebra
over $F$. A \textbf{magma} is a set with a binary multiplication.

The \textbf{free magma} $\mathcal F$ is defined inductively from an
alphabet $X=\{x_j : j\in J\}:$

$X \subset   \mathcal F;$  if $(y_1),(y_2) \in \mathcal F$ then
$(y_1y_2) \in \mathcal F.$ The elements $w\in \mathcal F$ are called
\textbf{words}. The \textbf{degree} $\deg_k(w)$ of $x_k$ in a word
$w$ is the number of times $x_i$ occurs in~$w$. Given a magma $M$,
one can construct the \textbf{magma algebra} $F[M]$ as the free
module $F^M$, endowed with
 multiplication
$$\left(\sum a_i w_i\right) \left(\sum \a_j' w_j\right)=  \sum _w \left(\sum_{w_iw_j = w} \a_i\a'_j\right) w.$$
A \textbf{division algebra} is an algebra (not necessarily
associative) in which every nonzero element is invertible.
 The
\textbf{free nonassociative algebra} $F\{x\}$ is the magma algebra
of the free magma in $X$.    A \textbf{nonassociative polynomial} is
an element $p =\sum \a_i w_i$ of $F\{x\}$. A \textbf{polynomial
identity} (PI) of a (not necessarily associative) algebra~$A$ over
$F$ is a noncommutative polynomial in the free nonassociative
algebra $F\{x\}$, which vanishes identically for any substitution
in~$A$. We write $\id (A)$ for the set of PI's of an algebra~$A$.
The set $\id(A)$ can be viewed as an ideal of $F\{x\}$, closed under
all algebra homomorphisms $F\{ x\}\to F\{ x\}$. Such an ideal $I$ of
$F\{ x\}$ is called a $T$-\textbf{ideal}. In general, the
\textbf{T-ideal} of a polynomial \textbf{in an algebra} $A$ is the
ideal generated by all substitutions of the polynomial in $A$.

Conversely, for any $T$-ideal $I$ of $F\{ x\}$, each element of $I$
is a PI of the quotient algebra $F\{ x\}/I$, and $F\{ x\}/I$ is
\textbf{relatively free}, in the sense that for any PI-algebra $A$
with $\id(A) \supseteq I,$ and any $a_j\in A,$ there is a unique
homomorphism $F\{ x\}/I \to A$ sending $x_j \mapsto a_j$ for $j \in
J$.

We need some background from \cite{BMS,Sp}. Let $A$ be an  algebra
over $F$ with an involution $(*)$ (anti-automorphism of degree $\le
2$) satisfying $$a +a^*, aa^*\in F,\quad \forall a \in A.$$ Then $A$
has a trace $\tra(a):= \frac 12 (a^*+a)$, and a norm $\|a\| =
a^*a.$\footnote{One can also treat characteristic 2 by defining
$\tra(a)$ to be the coefficient of $a$ in its minimal polynomial.
Also note that if $a = \alpha + n$ for $n^2 = 0$ then $a^2 =
\alpha^2.$}
 If $A = F,$ we take $(*)$ to be
the identity, so every norm is a square.

 Let $0\ne \a \in F$. The algebra $(A, \a)$ is defined to
have underlying vector space structure $ A \oplus A,$ with
componentwise addition and scalar multiplication, and multiplication
is defined by the formula determined by the formula
 $$(a_1, a_2)(b_1, b_2) = ( a_1b_1 + \a b_2^*a_2, b_2a_1 +
a_2b_1^* ). $$ There is a natural embedding $A \to (A, \a)$ given by
$a \mapsto (a,0)$. $(A, \a)$ in turn has an involution defined by
$(a_1, a_2)^* = (a_1^*,-a_2)$, and trace $\trace(a) =
\frac{a+a^*}2,$ so $\trace(a_1,a_2)= \trace(a_1).$ The norm is given
by
\begin{equation*} \begin{aligned} \| (a_1, a_2)\| & = (a_1, a_2)(a_1^*,- a_2) = (a_1 a_1^* - \a a_2^*
a_2, -a_2a_1 +a_2 a_1)\\ &= (\|a_1\|-\a \|a_2\|,0).\end{aligned}
\end{equation*}


When $A$ is a field, $(A,\a)$ is called a \textbf{quaternion
algebra} $Q$. As a special case,  take $F = \mathbb R$, $A = \mathbb
C,$ $(*)$ complex conjugation, and $(A, -1)$ is denoted as $\mathbb
H$, which is a division algebra. When $A$ is a quaternion algebra
$Q$ over an arbitrary field, $(Q,\a)$ is called a
\textbf{Cayley-Dickson algebra}.  $(Q,\a)$ is not associative, but
is alternative, i.e., satisfies the law $[a,a,b] = [a,b,a] = [b,a,a]
= 0$ for all $a,b$. For example, the \textbf{octonion division
algebra} $\mathcal O $ is $(\mathbb H, -1)$.

 $\mathcal O $ is a nonassociative
 division algebra.
 Every simple alternative algebra either is associative or a
 Cayley-Dickson algebra. Thus, the following result is an
 alternative version of \cite{Am}, which first appeared in \cite{Sh}, also cf.~\cite[Corollary~3.26]{Row2}, \cite[p.~22]{Row2}.

 \begin{theorem}\label{Am} The relatively free algebra of any Cayley-Dickson algebra is
 a central order in a
 Cayley-Dickson division
 algebra.
\end{theorem}
(It is clearly alternative, and satisfies the same polynomial
identities of the octonion division algebra, including its central
polynomial, so one uses the alternative version of the
Posner-Formanek-Markov-Razmyslov-Rowen theorem (see \cite[Corollary 1]{Mar})).

A quaternion algebra $Q$ can be written as $$Q = F+Fx+Fy+Fxy,$$
where $xu = -yx$; here $\trace(x) = \trace(y) = \trace(xy) = 0,$ and
$trace(1) = 2.$
 Now take  the
octonion algebra, $\Oct = (Q,\a)$ for an arbitrary quaternion
algebra~$Q$. There is a bilinear form $\langle \phantom{w} ,
\phantom{w} \rangle$ given by $\langle a,b \rangle = ab^*+ba^*$ for
$a,b\in\Oct.$ By \textbf{pure octonion} we mean an octonion of trace
0. We have
 an orthogonal base of
\textbf{basic octonions}, we mean the following eight
octonions:\begin{equation*}
   e_0 =1, e_1, e_2, e_3, e_4, e_5, e_6, e_7,
\end{equation*} where $e_2, \dots, e_8$ are pure octonions.

Every octonion $x$ admits a presentation in the
form:\begin{equation}
    x=\sum_{i=0}^7 \beta_ie_i
\end{equation}
for unique scalars $ \beta_i\in F$. The pure octonions are  those
octonions $x$ with $ \beta_0=0$, and the set of all pure octonions
will be denoted by $V$.\newline By the \textbf{scalar product}
$\langle \cdot, \cdot \rangle$ on $\mathbb{O}= (Q,\alpha)$, we mean
the  scalar product  $a\mapsto aa^*$. (This becomes the standard
Euclidean scalar product on the real octonions.) A \textbf{unit
length} octonion is an octonion $x$ with $\|x\|=1$. (In the case of
the octonion division algebra over $\R$, the scalar product is just
the usual scalar product on $\R^{(8)},$ and the $e_i$ can be taken
to be an orthonormal base.\newline

 By an \textbf{automorphism} of $A$ we mean
a vector space isomorphism respecting multiplication.

A subset $V\subset A$ is called:
\begin{enumerate}
    \item a \textbf{cone}, if it is closed under scalar multiples.
\item  \textbf{self similar}, if it is closed under automorphisms of $A$.
\item an \textbf{irreducible self similar cone}, if it is a self similar cone such that any subset of $V$ which is also a self similar cone is either $\{0\}$ or $V$.
 \end{enumerate}
 \begin{lemma}\label{image cone}
The image of a multilinear polynomial $p(x_1,\dots,x_n)$ on an algebra $A$ is a self similar cone.
\end{lemma}
\begin{proof}
The cone property follows from the fact that $p(\alpha a_1,a_2,\dots,a_n)=\alpha p(a_1,a_2,\dots,a_n)$ for scalar $\alpha$ and elements $a_i\in A$. Self similarity follows from the fact that $\varphi(p(a_1,a_2,\dots,a_n)=p(\varphi(a_1),\varphi(a_2),\dots,\varphi(a_n))$ for every automorphism $\varphi$ and elements $a_i\in A$.
\end{proof}

\section{Multilinear polynomial evaluations on O}
Assume throughout that $\a \in F$ satisfies the property:

\textbf{Property P}. $\|a_1-\a a_2\|$ is a square in $F$, for any
$a_i\in Q.$

One situation in which this occurs obviously is for $F$ closed under
square roots. Another situation is for when every norm in $F$ is a
square, $\a = -1,$ and $F$ is a Pythagorean field (the sum of two
squares is a square). This is the case for the real Octonions.

\begin{theorem}[{\cite[Corollary 1.7]{Sp}}]\label{automorphism}
Suppose Property P is satisfied. Let $x,y\in V$ be two nonzero pure
octonions. Then there exists $c\in F$ and an automorphism $\varphi
\colon \Oct \rightarrow \Oct $ such that $y=c\varphi(x)$. We can
take $c=1$ if $x$ and $y$  have the same norm.
\end{theorem}
\begin{proof} If $x,y$ have the same norm then this is  \cite[Corollary~1.7]{Sp}. In general, one can use Property P to find $c =
\sqrt{\|a_1\|-\|a_2\|\a}.$ let $\varphi_1$ be an automorphism of
$\Oct$ with $\varphi_1(i)=\frac{x}{\sqrt{\lVert x\rVert}}$ and let
$\varphi_2$ be an automorphism of $\Oct$ with
$\varphi_2(i)=\frac{y}{\sqrt{\lVert y\rVert}}$. If we define
$c=\frac{\lVert y\rVert}{\lVert x\rVert}$, then $\varphi=\varphi_2
\circ {\varphi_1}^{-1}$ is an automorphism of $\Oct$ with
$y=c\varphi(x)$.
\end{proof}
\begin{corollary}\label{ V irreducible}
The space $V$ of pure octonions is an irreducible self similar cone.
\end{corollary}
\begin{proof}
let $W\neq \{0\}$ be a self similar cone contained in $V$ and let
$w\in W$ be any nonzero element. If  $x\in V\setminus\{0\}$, there
exists an automorphism $\varphi \colon \Oct \rightarrow \Oct $ and
$c\in F $ such that $x=c\varphi(w)$. Since $W$ is a self similar
cone it follows that $x\in W$.
\end{proof}

\begin{lemma}\label{basic octonions}
 Let $p(x_1,\dots,x_n)\in  F \{x_1,\dots,x_n\}$ be a multilinear polynomial. If $e_{i_1},\dots,e_{i_n}$ are basic octonions, then the evaluation $p(e_{i_1},\dots,e_{i_n})$ is a scalar multiple of a basic octonion $q$ i.e. $p(e_{i_1},\dots,e_{i_n})=c\cdot q$, for some $c\in F $.
\end{lemma}
\begin{proof}
The product of two basic octonions is commutative up to sign, and
the product of three basic octonions is associative up to sign.
Write $p(x_1,\dots,x_n)$ as $p(x_1,\dots,x_n)=\sum_{i=1}^N
c_im_i(x_1,\dots,x_n)$, where $c_i\in F $ and each $m_i$ is a
monomial in $x_1,\dots,x_n$.  As $p(x_1,\dots,x_n)$ is multilinear,
each variable $x_j$  appears in every momomial $m_i$ exactly once.
If $q=m_1(e_{i_1},\dots,e_{i_n})$, then for any $i$,
$m_i(e_{i_1},\dots,e_{i_n})=b_i q$, where $b_i=\pm 1$. Thus
$p(e_{i_1},\dots,e_{i_n})=c\cdot q$ where $c=\sum_{i=1}^Nc_ib_i$.
\end{proof}

\begin{theorem}\label{Oct}
The image of a multilinear polynomial $p$ on an octonion algebra
$\Oct$ is either $\{0\}$, $ F $, $V$ or $\Oct$.
\end{theorem}
\begin{proof}
Consider the set of evaluations of $p$ on basic octonions.\newline
According to Lemma \ref{basic octonions} there are exactly four cases:\begin{enumerate}
    \item All evaluations give zero.
    \item All evaluations are in $F$, and there is a nonzero evaluation.
    \item All evaluations are pure octonions and there is a nonzero evaluation.
    \item There is an evaluation giving a nonzero element of $F$, and there is an evaluation giving a nonzero pure octonion.
\end{enumerate}
If (1) is the case, it follows that the image of $p$ on $\Oct$ is
 $\{0\}$. If (2) is the case, it follows that the image of $p$ on
$\Oct$ is $ F $. If (3) is the case, then since $\Image p\subset V$
and  $\Image p\neq\{0\}$, it follows from Lemma \ref{image cone} and
Lemma \ref{ V irreducible} that $\Image(p)=V$.\newline Thus it
suffices to prove that if (4) is the case then $\Image(p)=\Oct$.
There exist basic octonions $(x_1,\dots,x_n)$ and $(y_1,\dots,y_n)$
such that \begin{equation*}
    p(x_1,\dots,x_n)\in  F \setminus\{0\}
\end{equation*}  and \begin{equation*}
    p(y_1,\dots,y_n)\in V\setminus \{0\}.
\end{equation*} Let us define the following $n+1$ polynomial functions:

$A_0(z_1,\dots,z_n)=p(x_1,\dots,x_n)$ (a constant function),

$A_1(z_1)=p(z_1,x_2,\dots,x_n)$,

$A_2(z_1,z_2)=p(z_1,z_2,x_3,\dots,x_n)$,

 \noindent

 and so on, until

$A_n(z_1,\dots,z_n)=p(z_1,\dots,z_n)$.

Note that
$\Image A_0$ cosists of one nonzero scalar,
$\Image A_i\subset \Image A_{i+1}$ for
all $i$, and $\Image A_n=\Image p$ contains nonzero pure
octonions. Hence, there must
exist an $i$ such that $\Image A_{i-1}\subseteq  F $ and $\Image
A_{i}\nsubseteq  F $.\newline

We now claim that there exist octonions $o_1,\dots,o_n, \hat{o}_i$
such that $p(o_1,\dots,o_n)=r\in  F \setminus\{0\}$ and
$p(o_1,\dots,o_{i-1},\hat{o}_i,o_{i+1},\dots,o_n)\notin  F $. This
is proved as follows. Note that $p(z_1,\dots,z_n)$ induces $8$ real
valued polynomial functions of $8n$ real variables\begin{equation*}
    p(z_1,\dots,z_n)=p_0(t_1,\dots,t_{8n})e_0+\dots +p_7(t_1,\dots,t_{8n})e_7.
\end{equation*}
 This also induces
for each $j$, eight polynomial functions
\begin{equation*}
   A_j(z_1,\dots,z_j)= \Tilde{A}_{j,0}(t_1,\dots,t_{8j})e_0+\cdots+\Tilde{A}_{j,7}(t_1,\dots,t_{8j})e_7
\end{equation*} Choose $8i$ algebraically independent real numbers and form $i$ octonions by grouping eight numbers at a time to obtain $o_1,\dots,o_{i-1},\hat{o}_i$. Next define $o_i=x_i,o_{i+1}=x_{i+1},\dots,o_n=x_n$.  Since  $\Image A_{i-1}\subset  F $ we have $p(o_1,\dots,o_n)=A_{i-1}(o_1,\dots,o_{i-1})=r\in  F $. If $r=0$ were the case, then all eight polynomial functions $\Tilde{A}_{i-1,k}(t_1,\dots,t_{8i-8})$ would vanish at the algebraically independent set corresponding to $o_1,\dots,o_{i-1}$ and this will imply that $\Tilde{A}_{i-1,k}(t_1,\dots,t_{8i-8})=0$ identically,
 in contradiction with $p(x_1,\dots,x_n)\in  F \setminus\{0\}$.

Similarly, if
$p(o_1,\dots,\hat{o}_i,\dots,o_n)=A_i(o_1,\dots,\hat{o}_i)\in  F $,
this would imply that $\tilde{A}_{i,k}(t_1,\dots,t_{8i})$ for $k>1$
is zero, so that $\Image A_i\subset  F $, a contradiction.
\newline

Hence, let us choose octonions $o_1,\dots,o_n, \hat{o}_i$ such that
$p(o_1,\dots,o_n)=\beta\in  F \setminus\{0\}$ and
$p(o_1,\dots,o_{i-1},\hat{o}_i,o_{i+1},\dots,o_n)=\gamma +v$ where
$\gamma\in
 F $ and $v$ is a nonzero pure octonion. Define
$\Tilde{o}_i=\hat{o}_i-\beta^{-1}ao_i$ and note that by
multilinearity of $p$ we get
$p(o_1,\dots,o_{i-1},\tilde{o}_i,o_{i+1},\dots,o_n)=v$. Let $q$ be
an arbitrary octonion and write it as $q=b+u$ where $b$ is real and
$u$ is a pure octonion and we may assume that $u\neq 0$. By theorem
\ref{automorphism} there is a $c\in  F $ and an automorphism
$\varphi \colon \Oct \rightarrow \Oct $ such that $u=c\varphi(v)$.
Thus:\begin{multline*}
    p(\varphi(o_1),\dots,\varphi(\beta^{-1}bo_i+c\tilde{o}_i),\dots,\varphi(o_n))=
    \\
    =\varphi(p(o_1,\dots,\beta^{-1}bo_i+c\tilde{o}_i,\dots,o_n))=\varphi(b+cv)=b+u=q.
\end{multline*}
This proves that $q\in \Image p$ and thus $\Image p=\Oct$.
\end{proof}
\section{Semihomogeneous polynomial evaluations on O}
Here we  consider evaluations of semihomogeneous polynomials of
nonzero weighted degree on $\Oct$, and prove the following theorem:
\begin{theorem}\label{semi}
Let $p(x_1,\dots, x_m)$ be a semihomogeneous polynomial with nonzero
weighted degree $d$, i.e., letting $d_i$ be the degree of $x_i$,
there exist weights $w_1,w_2,\dots,w_m$ such that in each monomial
 $$w_1d_1+w_2d_2+\dots+w_md_m=d.$$ Then the evaluation of $p$ on
$\Oct$ is either $\{0\}$, or $F$, or V is Zariski dense in $\Oct$.
\end{theorem}
\begin{proof}
The set of polynomial evaluation on an algebra is closed under each
automorphism, and thus is self-similar. In analogy to Lemma
\ref{image cone}, the image of a semihomogeneous polynomial with
nonzero weighted degree is a   self similar cone. Here we cannot
consider only evaluations on basic octonions, since an evaluation of
a semihomogeneous polynomial does not have to be a linear
combination of its basic evaluations. Nevertheless, it follows from
Theorem~\ref{automorphism} that any two pure octonions having the
same norm are equivalent. Therefore the orbit of any octonion $a+v$
(where $a\in F$ and $v\in V$) is the set of all octonions $\{a+w:
\|w\|=\|v\|.\}$ Consider the function
$$f(x_1,\dots,x_m)=\frac{\|v\|}{a^2},$$ where
$p(x_1,\dots,x_m)=a+v,a\in F, v\in V.$ It is easy to see that this
function is a commutative rational function in the coordinates of
$x_i$. There are four possibilities:
\begin{enumerate}
\item The function $f$ is not a constant.
\item The function $f$ constantly equals $0$.
\item The function $f$ constantly equals $\infty$.
\item The function $f$ constantly equals some $c$ which is  neither $0$ nor~$\infty$.
\end{enumerate}
In the first case $\Image f$ must be Zariski dense and therefore
$\Image p$ must be Zariski dense as well.
 In the second case
$p$ takes on only scalar values and $\Image p=F.$ In the third case
$p$ takes only pure octonion values and $\Image p=V.$

If we can show that the fourth case is impossible, Theorem
\ref{semi} will be proved.

For that we consider the algebra of (central) fractions of
the ``Amitsur algebra'' of generic octonions.
 It is a free algebra in the variety generated by the algebra of Cayley numbers and, by \cite[Lemma 13.4]{ZSSS} or \cite{Row3}, does not contain zero divisors. The fact that this algebra is free is true for the algebra of generic elements of any finite-dimensional algebra (this fact is noted in \cite{Row2,Row3} and \cite[Proposition 1.2]{PS}).
As in the associative case, the scalar and pure parts of generic
octonions are elements of this algebra.

Note that each octonion $a+v$ has two eigenvalues
$$\lambda_{1}=a+\sqrt{-\|v\|}, \qquad \lambda_{2}=a -\sqrt{-\|v\|}.$$
 Therefore if a homogeneous polynomial
$p$ satisfying the fourth condition  exists, we will have the
situation where $\lambda_1(p)+\lambda_2(p)$ is an element of the
Amitsur algebra, both eigenvalue functions $\lambda_i(p)$ are linear
multiples of $\lambda_1(p)+\lambda_2(p)$ and thus are elements of
the Amitsur algebra as well. Each octonion satisfies the
Cayley-Hamilton polynomial, so
$$(p-\lambda_1(p))(p-\lambda_2(p))=0,$$
 contradicting  the fact that the
Amitsur algebra is a domain.
\end{proof}
\begin{remark}
If $p(x_1,\dots,x_m)$ is a semihomogeneous polynomial whose image
set $\Image p$ is dense in $\Oct$, then one considers the function
$f(x_1,\dots,x_m)$ which takes on the value
$\frac{\lambda_1}{\lambda_2}+\frac{\lambda_2}{\lambda_1}$ where
$\lambda_{1}$ and $\lambda_2$ are the eigenvalues of $p(x_1,\dots,x_m)$.
We call $\frac{\lambda_1}{\lambda_2}$ a ``ratio of eigenvalues.''

The function $f(x_1,\dots,x_m)$ is a rational function defined on
$8m$ parameters (being coordinates of $x_i$, each $x_i$ having $8$
coordinates). If $\Image p$ is dense in $\Oct$, then $\Image f$ must
be dense in $K$. If $K$ is algebraically closed, $\Image p$ must be
either $K$ itself or
 the complement of a finite set. Note that the value of $f$
is invariant while we remain in the  orbit of $p(x_1,\dots,x_m)$.
Therefore the image of $p$ is dense in $\Oct$ only if it contains
all the octonions except octonions having special ratios of
eigenvalues. Note that in case of ratio $1$ one can exclude either all octonions having equal eigenvlues, or scalar octonions only (the similar example for matrices was constructed in \cite[Example 4(i)]{BMR1}). Such a polynomial exists and can be constructed (see
Example \ref{expol} after this remark). If $K=\R$ then one can
construct examples similar to \cite[Example 1]{M2} using the fact
that any subalgebra of $\Oct$ generated by $2$ elements is
associative. The question of the full classification of possible
semihomogeneous dense evaluations on $\Oct$ and $\mathbb{H}$ over
$\R$ remains open.
\end{remark}
\begin{example}\label{expol}
Let $S$ be any finite subset of $K$. There exists a completely homogeneous
polynomial $p$ such that $\Image p$ is the set of all octonions except the
octonions with ratio of eigenvalues from $S$.
Consider
$$f(x)=x\cdot \prod_{\delta\in S}(\lambda_1-\lambda_2\delta)(\lambda_2-\lambda_1\delta) $$
This is a  ``polynomial
in $x$ with scalar parts and
norms'', the version of ``polynomial with traces'' from \cite{BMR1}.
Using the same construction described in \cite{BMR1} one can
construct a homogeneous polynomial which   equals  this polynomial
multiplied by a central polynomial. Therefore there exists a
homogeneous polynomial whose image is the set of all octonions
except those octonions whose ratio of eigenvalues are  from a finite
set $S$.
\end{example}

One more interesting example is the polynomial $p(x)=x^2$  not
taking nilpotent values. In order to classify all evaluations of
semihomogeneous polynomials one should resolve the following
question:
\begin{que}
Does there exist a semihomogeneous polynomial $p$ whose image set
$\Image p$ contains all the octonions except the unipotent
octonions? Recall that element $a+v$ (where $a$ is scalar and $v$ is
pure) is called unipotent if $a\neq 0$ and $v\neq 0$ is nilpotent.
\end{que}

Note that a polynomial from Example \ref{expol} does not take nilpotent values. Therefore we can announce one more question:

\begin{que}
Does there exist a polynomial that image set contains
all the octonions except octonions having special ratios of
eigenvalues and nilpotent octonions?
\end{que}

The same questions remain being open for $2\times 2$ matrices.

\section{evaluations of  nonassociative polynomials on anticommuting indeterminates, on Malcev algebras}
In this section we consider the Malcev algebra of pure octonions $V$
over $\R$, with product  defined as $vw=v\cdot w-w\cdot v,$ where
$\cdot$ is the standard octonion multiplication.

Obviously, this product satisfies the condition of
anticommutativity: $vw=-wv$, and, as well known, it satisfies the
Malcev identity:
$$(xy)(xz)=((xy)z)x+((yz)x)x+((zx)x)y.
$$
Thus, $V$ with this multiplication forms a Malcev algebra. In this
section we will consider the question of all possible polynomial
evaluations on~$V$, not only multilinear and/or semihomogeneous.
\begin{theorem}\label{Mal}
Let $p(x_1,\dots,x_m)$ be an arbitrary polynomial in $m$
anticommuting variables. Then its evaluation on $V$ is either
$\{0\}$ or $V$.
\end{theorem}
\begin{proof}
Indeed, if we consider the square of the norm of $p(x_1,\dots,x_m)$,
it is a commutative polynomial in $7m$ coordinates of $x_i$. (Each
of them has seven coordinates since its real part is $0$.)
Therefore, we have two cases: either this function is constant or
not. If it is a constant it must equal zero since $p(0,0,\dots,0)=0$
and each value of $p$ must have norm $0$, i.e. equals $0$. Assume
that the norm is not a constant. Thus (recalling that it is a
polynomial) for any $M$ it takes values greater than $M$, and being
a continuous function, it satisfies the intermediate value theorem
and   takes on all non-negative values. Thus, for any $M\geq 0$
there exist $x_1,\dots,x_m$ such that $\|p(x_1,\dots,x_m)\|^2=M$.
By~Theorem~\ref{automorphism}, any two pure octonions with equal
norm are equivalent. Equivalence in the Malcev algebra follows from
equivalence in $\Oct$. Therefore, $\Image p=V.$
\end{proof}


\begin{thebibliography}{9}
\bibitem{Am} Amitsur, S.A., {\em The T-ideals of the free ring}, J.
London Math Soc. 30 (1955), 464--470.
\bibitem{BMR1}  Belov, A.;  Malev, S.; Rowen, L. {\em The images of non-commutative polynomials evaluated on
$2\times 2$ matrices},
Proc. Amer. Math. Soc {\bf 140} (2012), 465--478.

\bibitem{BMR2} Kanel-Belov, A.; Malev, S.; Rowen, L.
{\it The images of multilinear polynomials evaluated on $3\times 3$ matrices
}
Proc. Amer. Math. Soc. {\bf 144} (2016), no. 1, 7--19.
\bibitem{BMR3} Kanel-Belov, A.; Malev, S.; Rowen, L.

{\it Power central polynomials on matrices
}
Journal of Pure and Applied Algebra {\bf 220} (2016), no. 6, 2164$-$2176.
\bibitem{BMR4} Kanel-Belov, A.; Malev, S.; Rowen, L.

{\it Evaluations of Lie Polynomials on $2\times 2$ Matrices}
Communications in Algebra  {\bf 45} (2017), issue 11, 4801$-$4808.
\bibitem{BMRY} Kanel-Belov, Alexei; Malev, Sergey; Rowen, Louis; Yavich, Roman {\em Evaluations of noncommutative polynomials on algebras: methods and problems, and the L'vov-Kaplansky conjecture.} SIGMA Symmetry Integrability Geom. Methods Appl. 16 (2020), Paper No. 071, 61 pp.

\bibitem{BMS} Bremner, M.R.,   Murakami L.I., and  Shestakov, I.P.,
{\it Nonassociative Algebras} (2007), DOI: 10.1201/b16113-105.

\bibitem{Br} Bre\v{s}ar, M., Commutators and images of noncommutative polynomials, Advances in Mathematics 374 (2020), 107346


\bibitem{BrK} Bre\v{s}ar, M., Klep, I., A note on values of noncommutative polynomials, Proc. Amer. Math. Soc. 138 (2010), 2375-2379.

\bibitem{Ch} Chuang C.-L., On ranges of polynomials in finite matrix rings, Proc.
Amer. Math. Soc. 110 (1990), 293--302.

\bibitem{DK} Dykema K.J., Klep I., Instances of the Kaplansky-Lvov multilinear
conjecture for polynomials of degree three, Linear Algebra Appl. 508
(2016), 272-288,

\bibitem{LT} Li C., Tsui M.C., On the images of multilinear maps of matrices over
finite-dimensional division algebras, Linear Algebra Appl. 493
(2016), 399--410.

\bibitem{MO} Ma A., Oliva J., On the images of Jordan polynomials evaluated over
symmetric matrices, Linear Algebra Appl. 492 (2016), 13--25.

\bibitem {M1} Malev, S.
{\it The images of non-commutative polynomials evaluated on $2\times 2$ matrices over an arbitrary field
}
Journal of Algebra and its Applications {\bf 13} (2014), no. 6, Paper No. 1450004, 12 pp.

\bibitem{M2} Malev, S., {\em The images of noncommutative polynomials evaluated on the Quaternion algebra,}
Journal of Algebra and its Applications (2021), no. 5, Paper No. 2150074, 8 pp.
\bibitem{MP}
Malev, S.; Pines, C. {\em The images of multilinear non-associative
polynomials evaluated on a rock-paper-scissors algebra with unit
over an arbitrary field and its subalgebras.} Chebyshevski\u{i} Sb.
21 (2020), no. 4, 129--139.

\bibitem{ML} Makar-Limanov L., {\em An example of a skew field without a trace}, Comm. Algebra~17 (1989), 2303--2307.

\bibitem{Mar} Markov, V.T., {\it On the dimension of noncommutative affine algebras},
Izv. Math., 7:2 (1973), 281-285.

\bibitem{PS} Polikarpov, S.V., Shestakov, I. P.{\em
Nonassociative affine algebras.} (English. Russian original) Algebra
Logic 29, No. 6, 458--466 (1990); translation from Algebra Logika
29, No. 6, 709--723 (1990).

\bibitem{Row1} Rowen, L.H., {\it Polynomial identities of nonassociative rings, Ill}.
Journal Math. 22 (1978), 342--378.


\bibitem{Row2} Rowen, L.H., {\it Polynomial identities of nonassociative rings II, Ill}.
Journal Math. 22 (1978), 521--540.

\bibitem{Row3} Rowen, L.H., {\it Polynomial identities of nonassociative rings III, Ill}.
Journal Math. 23 (1979), 15--35.



\bibitem{Sch} Schafer, R.D.,   An Introduction to Nonassociative Algebras,
corrected reprint of the 1966 original, Dover Publications, New
York, 1995. MR1375235

\bibitem{Sh} Shestakov, I.P.,
Radicals and nilpotent elements of free alternative algebras.
(Russian) Algebra i Logika 14 (1975), no. 3, 354--365, 370.


\bibitem{Spe} \v{S}penko, \v{S}., On the image of a noncommutative polynomial, J. Algebra 377 (2013), 298--311.


\bibitem{Sp} Springer, T.A., and Veldkamp,  F.D., {\em Octonions, Jordan Algebras and Exceptional Groups}. Springer 2000.

\bibitem{W1} Wang Y., The images of multilinear polynomials on $2 \times 2$ upper
triangular matrix algebras, Linear Multilinear Algebra 67 (2019),
2366--2372.

\bibitem{W2} Wang Y., The image of arbitrary polynomials on $2 \times 2$ upper triangular
matrix algebras, Preprint, 2019.

\bibitem{WLB} Wang Y., Liu P.P., Bai J., Correction: ''The images of multilinear
polynomials on $2 \times 2$ upper triangular matrix algebras'',
Linear Multilinear Algebra 67 (2019), no. 11, i--vi.

\bibitem{ZSSS}
Zhevlakov, K.; Slinko, A.; Shestakov, I.; Shirshov, A. {\em Rings close to associative}, Moscow, Nauka (1978), 432 pp.

\end{thebibliography}
\end{document}